\documentclass[12pt,reqno]{amsart}
\usepackage{amsmath, amsfonts, amssymb, amsthm, hyperref}
\usepackage{bm}
\allowdisplaybreaks[4]
\def\l{\left}
\def\r{\right}
\def\bg{\bigg}
\def\({\bg(}
\def\){\bg)}
\def\t{\text}
\def\f{\frac}

\def\sgn{{\rm sgn}}

\def\eq{\equiv}

\def\Z{\mathbb Z}

\def\N{\mathbb N}

\def\1{{\bf 1}}

\theoremstyle{plain}
\newtheorem{theorem}{Theorem}[section]
\newtheorem{lemma}{Lemma}

\newtheorem{conjecture}{Conjecture}

\theoremstyle{definition}

\def\<{\langle}
\def\>{\rangle}

\begin{document}
\hbox{}
\medskip

\title{Two congruences concerning Ap\'{e}ry numbers}
\author{Chen Wang}
\address {(Chen Wang) Department of Mathematics, Nanjing
University, Nanjing 210093, People's Republic of China}
\email{cwang@smail.nju.edu.cn}

\thanks{This work was supported by the National Natural Science Foundation of China (grant no. 11971222)}
 \begin{abstract} Let $n$ be a nonnegative integer. The $n$-th Ap\'{e}ry number is defined by
 $$
 A_n:=\sum_{k=0}^n\binom{n+k}{k}^2\binom{n}{k}^2.
 $$
Z.-W. Sun ever investigated the congruence properties of Ap\'{e}ry numbers and posed some conjectures. For example, Sun conjectured that for any prime $p\geq7$
 $$
 \sum_{k=0}^{p-1}(2k+1)A_k\equiv p-\frac{7}{2}p^2H_{p-1}\pmod{p^6}
 $$
 and for any prime $p\geq5$
 $$
 \sum_{k=0}^{p-1}(2k+1)^3A_k\equiv p^3+4p^4H_{p-1}+\frac{6}{5}p^8B_{p-5}\pmod{p^9},
 $$
 where $H_n=\sum_{k=1}^n1/k$ denotes the $n$-th harmonic number and $B_0,B_1,\ldots$ are the well-known Bernoulli numbers. In this paper we shall confirm these two conjectures.

\end{abstract}
\maketitle

\section{Introduction}
\setcounter{lemma}{0}
\setcounter{theorem}{0}
\setcounter{equation}{0}
\setcounter{conjecture}{0}
\setcounter{proposition}{0}

The well-known Ap\'{e}ry numbers given by
$$
A_n:=\sum_{k=0}^n\binom{n+k}{k}^2\binom{n}{k}^2=\sum_{k=0}^n\binom{n+k}{2k}^2\binom{2k}{k}^2\quad (n\in\N=\{0,1,\ldots\}),
$$
were first introduced by Ap\'{e}ry to prove the irrationality of $\zeta(3)=\sum_{n=1}^{\infty}1/n^3$ (see \cite{Apery,Sl}).

In 2012, Z.-W. Sun introduced the Ap\'{e}ry polynomials
$$
A_n(x)=\sum_{k=0}^n\binom{n+k}{k}^2\binom{n}{k}^2x^k\quad (n\in\N)
$$
and deduced various congruences involving sums of such polynomials. (Clearly, $A_n(1)=A_n$.) For example, for any odd prime $p$ and integer $x$, he obtained that
\begin{equation}\label{sunresult1}
\sum_{k=0}^{p-1}(2k+1)A_k(x)\eq p\l(\f{x}{p}\r)\pmod{p^2},
\end{equation}
where $(-)$ denotes the Legendre symbol.
Letting $x=1$ and for any prime $p\geq5$, Sun established the following generalization of \eqref{sunresult1}:
\begin{equation}\label{sunresult2}
\sum_{k=0}^{p-1}(2k+1)A_k\eq p+\f{7}{6}p^4B_{p-3}\pmod{p^5},
\end{equation}
where $B_0,B_1,\ldots$ are the well-known Bernoulli numbers defined as follows:
 $$
 B_0=0,\sum_{k=0}^{n-1}\binom{n}{k}B_k=0\quad (n=2,3,\ldots).
 $$
In 1850 Kummer (cf. \cite{IR}) proved that for any odd prime $p$ and any even number $b$ with $b\not\eq0\pmod{p-1}$
\begin{equation}\label{kummer1}
\f{B_{k(p-1)+b}}{k(p-1)+b}\eq\f{B_b}{b}\pmod{p}\quad\t{for}\ k\in\N.
\end{equation}

For $m\in\Z^+=\{1,2,\ldots\}$ the $n$-th harmonic numbers of order $m$ are defined by
$$H_n^{(m)}:=\sum_{k=1}^n\frac1{k^m}\quad (n=1,2,\ldots)$$
and $H_0^{(m)}:=0$. For the sake of convenience we often use $H_n$ instead of $H_n^{(1)}$. From \cite{Glaisher} we know that $H_{p-1}\eq-p^2B_{p-3}/3\pmod{p^3}$ for any prime $p\geq5$. Thus \eqref{sunresult2} has the following equivalent form
\begin{equation}\label{sunresult3}
 \sum_{k=0}^{p-1}(2k+1)A_k\equiv p-\frac{7}{2}p^2H_{p-1}\pmod{p^5}.
\end{equation}
Via some numerical computation, Sun \cite[Conjecture 4.2]{Sunapery} conjectured that \eqref{sunresult3} also holds modulo $p^6$ provided that $p\geq7$. This is our first theorem.
\begin{theorem}\label{th1}
For any prime $p\geq7$
\begin{equation}\label{thm1}
\sum_{k=0}^{p-1}(2k+1)A_k\equiv p-\frac{7}{2}p^2H_{p-1}\pmod{p^6}.
\end{equation}
\end{theorem}

Motivated by Sun's work on Ap\'{e}ry polynomials, V.J.W. Guo and J. Zeng studied the divisibility of the following sums:
$$
\sum_{k=0}^{n-1}(2k+1)^{2r+1}A_k\quad (n\in\Z^+\ \t{and}\ r\in\N).
$$
Particularly, for $r=1$, they obtained
\begin{equation}
\label{modn3}\sum_{k=0}^{n-1}(2k+1)^3A_k\eq0\pmod{n^3}
\end{equation}
and
\begin{equation}
\label{modp6}\sum_{k=0}^{p-1}(2k+1)^3A_k\eq p^3\pmod{2p^6},
\end{equation}
where $p\geq5$ is a prime. As an extension to \eqref{modp6}, Sun \cite[Conjecture A65]{Sunconj} proposed the following challenging conjecture.
\begin{conjecture}\label{sunconj}For any prime $p\geq5$ we have
\begin{equation*}
\sum_{k=0}^{p-1}(2k+1)^3A_k\eq p^3+4p^4H_{p-1}+\f{6}{5}p^8B_{p-5}\pmod{p^9}.
\end{equation*}
\end{conjecture}
This is our second theorem.
\begin{theorem}\label{th2}
Conjecture \ref{sunconj} is true.
\end{theorem}

Proofs of Theorems \ref{th1}--\ref{th2} will be given in Sections 2--3 respectively.
\section{Proof of Theorem \ref{th1}}
\setcounter{lemma}{0}
\setcounter{theorem}{0}
\setcounter{equation}{0}
\setcounter{conjecture}{0}
\setcounter{proposition}{0}
The proofs in this paper strongly depend on the congruence properties of harmonic numbers and the Bernoulli numbers. (The readers may consult \cite{IR,zhber,sunharmonic,sz} for the properties of them.) Below we first list some congruences involving harmonic numbers and the Bernoulli numbers which may be used later.

\begin{lemma}\label{cms}\cite[Remark 3.2]{CMS} For any prime $p\geq5$ we have
$$
2H_{p-1}+pH_{p-1}^{(2)}\eq\f{2}{5}p^4B_{p-5}\pmod{p^5}.
$$
\end{lemma}

From \cite[Theorems 5.1\&5.2]{zhber}, we have the following congruences.
\begin{lemma}\label{harmonic} For any prime $p\geq7$ we have
\begin{gather*}
H_{(p-1)/2}\eq-2q_p(2)\pmod{p},\\
H_{p-1}^{(2)}\eq\l(\f{4}{3}B_{p-3}-\f{1}{2}B_{2p-4}\r)p+\l(\f{4}{9}B_{p-3}-\f{1}{4}B_{2p-4}\r)p^2\pmod{p^3},\\
H_{(p-1)/2}^{(2)}\eq\l(\f{14}{3}B_{p-3}-\f{7}{4}B_{2p-4}\r)p+\l(\f{14}{9}B_{p-3}-\f{7}{8}B_{2p-4}\r)p^2\pmod{p^3},\\
H_{p-1}^{(3)}\eq-\f{6}{5}p^2B_{p-5}\pmod{p^3},\quad H_{(p-1)/2}^{(3)}\eq6\l(\f{2B_{p-3}}{p-3}-\f{B_{2p-4}}{2p-4}\r)\pmod{p^2},\\
H_{p-1}^{(4)}\eq\f{4}{5}pB_{p-5}\pmod{p^2},\quad H_{(p-1)/2}^{(4)}\eq0\pmod{p},\quad H_{p-1}^{(5)}\eq0\pmod{p^2}.
\end{gather*}
where $q_p(2)$ denotes the Fermat quotient $(2^{p-1}-1)/p$.
\end{lemma}
\noindent {\it Remark 2.1}.
By Kummer's congruence \eqref{kummer1}, we know $B_{2p-4}\eq4B_{p-3}/3\pmod{p}$. Then the congruences of $H_{p-1}^{(2)}$ and $H_{(p-1)/2}^{(2)}$ can be reduced to
$$
H_{p-1}^{(2)}\eq\l(\f{4}{3}B_{p-3}-\f{1}{2}B_{2p-4}\r)p+\f{1}{9}p^2B_{p-3}\pmod{p^3}
$$
and
$$
H_{(p-1)/2}^{(2)}\eq\l(\f{14}{3}B_{p-3}-\f{7}{4}B_{2p-4}\r)p+\f{7}{18}p^2B_{p-3}\pmod{p^3}
$$
respectively.
By Lemma \ref{cms}, we immediately obtain that $H_{p-1}\eq-pH_{p-1}^{(2)}/2\pmod{p^4}$. Thus
\begin{equation}\label{h1mod4}
H_{p-1}\eq\l(\f{1}{4}B_{2p-4}-\f{2}{3}B_{p-3}\r)p^2-\f{1}{18}p^3B_{p-3}\pmod{p^4}.
\end{equation}

Recall that the Bernoulli polynomials $B_n(x)$ are defined as
\begin{equation}\label{berdef1}
B_n(x):=\sum_{k=0}^n\binom{n}{k}B_kx^{n-k}\quad (n\in\N).
\end{equation}
Clearly, $B_n=B_n(0)$. Also, we have
\begin{equation}\label{berdef2}
\sum_{k=1}^{n-1}k^{m-1}=\f{B_m(n)-B_m}{m}
\end{equation}
for any positive integer $n$ and $m$. 

Let $d>0$ and ${\bf s}:=(s_1,\ldots,s_d)\in(\Z\backslash\{0\})^d$. The alternating multiple harmonic sum \cite{TZ} is defined as follows
$$
H({\bf s};n):=\sum_{1\leq k_1<k_2<\cdots k_d\leq n}\prod_{i=1}^d\f{\sgn(s_i)^{k_i}}{k_i^{|s_i|}}.
$$
Clearly, $H_n^{(m)}=H(m;n)$.

Let $A,B,D,E,F$ be defined as in \cite[Section 6]{TZ}, i.e.,
\begin{gather*}
A:=\sum_{k=2}^{p-3}B_kB_{p-3-k},\quad B:=\sum_{k=2}^{p-3}2^kB_kB_{p-3-k},\quad D:=\sum_{k=2}^{p-3}\f{B_kB_{p-3-k}}{k},\\
\quad E:=\sum_{k=2}^{p-3}\f{2^kB_kB_{p-3-k}}{k},\quad F:=\sum_{k=2}^{p-3}\f{2^{p-3-k}B_kB_{p-3-k}}{k}.
\end{gather*}

\begin{lemma}\label{proplemma1}
For any prime $p\geq7$ we have
$$
D-4F\eq 2B-2A-q_p(2)B_{p-3}\pmod{p}.
$$
\end{lemma}
\begin{proof}
In \cite[Section 6]{TZ}, Tauraso and Zhao proved that
$$
H(1,-3;p-1)\eq B-A\eq 2E-2D+2q_p(2)B_{p-3}\pmod{p}
$$
and
$$
\f{5}{2}D-2E-2F-\f{3}{2}q_p(2)B_{p-3}\eq0\pmod{p}.
$$
Combining the above two congruences we immediately obtain the desired result.
\end{proof}

\begin{lemma}\label{proph31}
Let $p\geq7$ be a prime. Then we have
\begin{equation}
H(3,1;(p-1)/2)\eq H_{(p-1)/2}^{(3)}H_{(p-1)/2}-4B+4A\pmod{p}.
\end{equation}
\end{lemma}
\begin{proof} By Lemma \ref{harmonic}, it is easy to check that
\begin{equation}\label{h13h31}
\begin{aligned}
H_{(p-1)/2}^{(3)}H_{(p-1)/2}=&H(1,3;(p-1)/2)+H(3,1;(p-1)/2)+H_{(p-1)/2}^{(4)}\\
\eq& H(1,3;(p-1)/2)+H(3,1;(p-1)/2)\pmod{p}.
\end{aligned}
\end{equation}
Thus it suffices to evaluate $H(1,3;(p-1)/2)$ modulo $p$. By Fermat's little theorem, \eqref{berdef1} and \eqref{berdef2} we arrive at
\begin{align*}
H(1,3;(p-1)/2)=&\sum_{1\leq j<k\leq (p-1)/2}\f{1}{jk^3}\eq\sum_{1\leq j<k\leq (p-1)/2}\f{j^{p-2}}{k^3}=\sum_{1\leq k\leq (p-1)/2}\f{B_{p-1}(k)-B_{p-1}}{k^3(p-1)}\\
=&\sum_{1\leq k\leq (p-1)/2}\f{\sum_{i=1}^{p-1}\binom{p-1}{i}k^{i-3}B_{p-1-i}}{p-1}=\sum_{i=1}^{p-1}\f{\binom{p-1}{i}B_{p-1-i}}{p-1}\sum_{k=1}^{(p-1)/2}k^{i-3}\\
\eq&\f{\binom{p-1}{2}B_{p-3}}{p-1}H_{(p-1)/2}+\sum_{i=4}^{p-1}\f{\binom{p-1}{i}B_{p-1-i}}{p-1}\cdot\f{B_{i-2}\l(\f{1}{2}\r)-B_{i-2}}{i-2}\pmod{p},
\end{align*}
where the last step follows from the fact $B_n=0$ for any odd $n\geq3$. By \cite{IR} we know that $B_n(1/2)=(2^{1-n}-1)B_n$. Thus
\begin{align*}
H(1,3;(p-1)/2)\eq&-B_{p-3}H_{(p-1)/2}-\sum_{i=4}^{p-1}\f{(2^{3-i}-2)B_{p-1-i}B_{i-2}}{i-2}\\
=&-B_{p-3}H_{(p-1)/2}-\sum_{i=2}^{p-3}\f{(2^{1-i}-2)B_{p-3-i}B_{i}}{i}\\
\eq&-B_{p-3}H_{(p-1)/2}-8F+2D\pmod{p}.
\end{align*}
With helps of Lemmas \ref{harmonic} and \ref{proplemma1}, we have
$$
H(1,3;(p-1)/2)\eq4B-4A\pmod{p}.
$$
Combining this with \eqref{h13h31}, we have completed the proof of Lemma \ref{proph31}.
\end{proof}

\begin{lemma}\label{h21}Let $p\geq7$ be a prime. Then we have
$$
\sum_{k=1}^{(p-1)/2}\f{H_k^{(2)}}{k}\eq\f{3}{2p^2}H_{p-1}+\f{1}{2}H_{(p-1)/2}^{(3)}+\f{1}{2}H_{p-1}^{(2)}H_{(p-1)/2}-pH_{(p-1)/2}^{(3)}H_{(p-1)/2}+4p(B-A)\pmod{p}.
$$
\end{lemma}
\begin{proof}
By \cite[Eq. (3.13)]{LW} we know that for any odd prime $p$
\begin{equation}\label{key'}
\sum_{k=1}^{p-1}\f{H_k^{(2)}}{k}\eq\f{3}{p^2}H_{p-1}\pmod{p^2}.
\end{equation}
On the other hand,
\begin{align*}
\sum_{k=1}^{p-1}\f{H_k^{(2)}}{k}=&\sum_{k=1}^{(p-1)/2}\f{H_k^{(2)}}{k}+\sum_{k=1}^{(p-1)/2}\f{H_{p-k}^{(2)}}{p-k}.
\end{align*}
For $k=1,2,\ldots,(p-1)/2$ we have
\begin{align*}
H_{p-k}^{(2)}=&\sum_{j=k}^{p-1}\f{1}{(p-j)^2}\eq\sum_{j=k}^{p-1}\l(\f{1}{j^2}+\f{2p}{j^2}\r)\eq H_{p-1}^{(2)}-H_{k-1}^{(2)}-2pH_{k-1}^{(3)}\pmod{p^2}
\end{align*}
by Lemma \ref{harmonic}. Thus
\begin{align*}
\sum_{k=1}^{p-1}\f{H_k^{(2)}}{k}\eq&2\sum_{k=1}^{(p-1)/2}\f{H_k^{(2)}}{k}-H_{(p-1)/2}^{(3)}+p\sum_{k=1}^{(p-1)/2}H(2,2;(p-1)/2)\\
&-H_{p-1}^{(2)}H_{(p-1)/2}+2pH(3,1;(p-1)/2)\pmod{p^2}.
\end{align*}
In view of Lemma \ref{harmonic}, we have
$$H(2,2;(p-1)/2)=\f{\l(H_{(p-1)/2}^{(2)}\r)^2}{2}-\f{H_{(p-1)/2}^{(4)}}{2}\eq0\pmod{p}.$$
This together with Lemma \ref{proph31} proves Lemma \ref{h21}. 
\end{proof}

\begin{lemma}\cite[Lemma 2.1]{Sunapery}\label{identity1} Let $k\in\N$. Then for $n\in\Z^+$ we have
$$
\sum_{m=0}^{n-1}(2m+1)\binom{m+k}{2k}^2=\f{(n-k)^2}{2k+1}\binom{n+k}{2k}^2.
$$
\end{lemma}

\noindent {\it Proof of Theorem \ref{th1}}. By Lemma \ref{identity1} it is routine to check that
\begin{align*}
\sum_{m=0}^{p-1}(2m+1)A_m=&\sum_{m=0}^{p-1}(2m+1)\sum_{k=0}^m\binom{m+k}{2k}^2\binom{2k}{k}^2=\sum_{k=0}^{p-1}\binom{2k}{k}^2\sum_{m=0}^{p-1}(2m+1)\binom{m+k}{2k}^2\\
=&\sum_{k=0}^{p-1}\binom{2k}{k}^2\f{(p-k)^2}{2k+1}\binom{p+k}{2k}^2=p^2\sum_{k=0}^{p-1}\f{1}{2k+1}\binom{p-1}{k}^2\binom{p+k}{k}^2.
\end{align*}
Note that
\begin{align*}
\binom{p-1}{k}^2\binom{p+k}{k}^2=&\prod_{j=1}^k\l(1-\f{p^2}{j^2}\r)^2\eq\prod_{j=1}^k\l(1-\f{2p^2}{j^2}+\f{p^4}{j^4}\r)\\
\eq&1-2p^2H_k^{(2)}+p^4H_k^{(4)}+4p^4H(2,2;k)\pmod{p^5}.
\end{align*}
Since $H_{(p-1)/2}^{(4)}\eq0\pmod{p}$ and $H(2,2;(p-1)/2)\eq0\pmod{p}$, we have
\begin{align}\label{key}
\sum_{m=0}^{p-1}(2m+1)A_m\eq p^2\Sigma_1-2p^4\Sigma_2\pmod{p^6},
\end{align}
where
$$
\Sigma_1:=\sum_{k=0}^{p-1}\f{1}{2k+1}\quad \t{and}\quad \Sigma_2:=\sum_{k=0}^{p-1}\f{H_k^{(2)}}{2k+1}.
$$

We first consider $\Sigma_1$ modulo $p^4$. Clearly,
\begin{align*}
\sum_{k=(p+1)/2}^{p-1}\f{1}{2k+1}=&\sum_{k=0}^{(p-3)/2}\f{1}{2(p-1-k)+1}\\
\eq&-8p^3\sum_{k=0}^{(p-3)/2}\f{1}{(2k+1)^4}-2p\sum_{k=0}^{(p-3)/2}\f{1}{(2k+1)^2}\\
&-4p^2\sum_{k=0}^{(p-3)/2}\f{1}{(2k+1)^3}-\sum_{k=0}^{(p-3)/2}\f{1}{2k+1}\pmod{p^4}.
\end{align*}
For $r\in\{2,3,4\}$,
$$
\sum_{k=0}^{(p-3)/2}\f{1}{(2k+1)^r}=H_{p-1}^{(r)}-\f{1}{2^r}H_{(p-1)/2}^{(r)}.
$$
By the above and in view of Lemma \ref{harmonic},
\begin{equation}\label{Sigma1}
\begin{aligned}
\Sigma_1=&\f{1}{p}+\sum_{k=0}^{(p-3)/2}\f{1}{2k+1}+\sum_{k=(p+1)/2}^{p-1}\f{1}{2k+1}\\
\eq&\f{1}{p}-2p\l(H_{p-1}^{(2)}-\f{1}{4}H_{(p-1)/2}^{(2)}\r)+\f{1}{2}p^2H_{(p-1)/2}^{(3)}\pmod{p^4}.
\end{aligned}
\end{equation}

Now we turn to $\Sigma_2$ modulo $p^2$. By Lemma \ref{harmonic},
\begin{align*}
&\sum_{k=(p+1)/2}^{p-1}\f{H_k^{(2)}}{2k+1}=\sum_{k=0}^{(p-3)/2}\f{H_{p-1-k}^{(2)}}{2(p-1-k)+1}\\
\eq&\sum_{k=0}^{(p-3)/2}\f{H_k^{(2)}}{2k+1}+2p\sum_{k=0}^{(p-3)/2}\f{H_k^{(2)}}{(2k+1)^2}+\f{1}{2}H_{p-1}^{(2)}H_{(p-1)/2}+2p\sum_{k=0}^{(p-3)/2}\f{H_k^{(3)}}{2k+1}\pmod{p^2}.
\end{align*}
Thus
$$
\Sigma_2\eq \f{H_{(p-1)/2}^{(2)}}{p}+2\sigma_1+\f{1}{2}H_{p-1}^{(2)}H_{(p-1)/2}+2p\sigma_2\pmod{p^2},
$$
where
$$
\sigma_1:=\sum_{k=0}^{(p-3)/2}\f{H_k^{(2)}}{2k+1}+p\sum_{k=0}^{(p-3)/2}\f{H_k^{(2)}}{(2k+1)^2}
$$
and
$$
\sigma_2:=\sum_{k=0}^{(p-3)/2}\f{H_k^{(3)}}{2k+1}.
$$
It is easy to see that
\begin{align*}
\sigma_1\eq&-\sum_{k=0}^{(p-3)/2}\f{H_k^{(2)}}{p-1-2k}=-\sum_{k=1}^{(p-1)/2}\f{H_{(p-1)/2-k}^{(2)}}{2k}\\
\eq&-\f{1}{2}H_{(p-1)/2}H_{(p-1)/2}^{(2)}+\f{1}{2}\sum_{k=1}^{(p-1)/2}\f{1}{k}\sum_{k=0}^{k-1}\l(\f{4}{(2j+1)^2}+\f{8p}{(2j+1)^3}\r)\\
=&-\f{1}{2}H_{(p-1)/2}H_{(p-1)/2}^{(2)}+\f{1}{2}\sum_{k=1}^{(p-1)/2}\f{1}{k}\l(4H_{2k}^{(2)}-H_k^{(2)}+8pH_{2k}^{(3)}-pH_k^{(3)}\r)\pmod{p^2}.
\end{align*}
Also,
\begin{align*}
\sigma_2\eq&-\sum_{k=0}^{(p-3)/2}\f{H_k^{(3)}}{p-1-2k}=-\sum_{k=1}^{(p-1)/2}\f{H_{(p-1)/2-k}^{(3)}}{2k}\\
\eq&-\f{1}{2}H_{(p-1)/2}H_{(p-1)/2}^{(3)}+\f{1}{2}\sum_{k=1}^{(p-1)/2}\f{1}{k}\sum_{j=0}^{k-1}\f{-8}{(2j+1)^3}\\
=&-\f{1}{2}H_{(p-1)/2}H_{(p-1)/2}^{(3)}-4\sum_{k=1}^{(p-1)/2}\f{1}{k}\l(H_{2k}^{(3)}-\f{1}{8}H_k^{(3)}\r)\pmod{p}.
\end{align*}
Combining the above we deduce that
\begin{equation}\label{Sigma2-1}
\begin{aligned}
\Sigma_2\eq&\f{H_{(p-1)/2}^{(2)}}{p}-H_{(p-1)/2}H_{(p-1)/2}^{(2)}+4\sum_{k=1}^{(p-1)/2}\f{H_{2k}^{(2)}}{k}-\sum_{k=1}^{(p-1)/2}\f{H_{k}^{(2)}}{k}\\
&+\f{1}{2}H_{p-1}^{(2)}H_{(p-1)/2}-pH_{(p-1)/2}H_{(p-1)/2}^{(3)}\pmod{p^2}.
\end{aligned}
\end{equation}
Note that
\begin{equation}\label{H2kk}
\sum_{k=1}^{(p-1)/2}\f{H_{2k}^{(2)}}{k}=\sum_{k=1}^{p-1}\f{H_k^{(2)}}{k}+H(2,-1;p-1)+\f{1}{4}H_{(p-1)/2}^{(3)}-H_{p-1}^{(3)}.
\end{equation}
By \cite[Proposition 7.3]{TZ} we know that
\begin{equation}\label{h2-1}
H(2,-1;p-1)\eq-\f{3}{2}X-\f{7}{6}pq_p(2)B_{p-3}+p(B-A)\pmod{p^2},
\end{equation}
where $X:=B_{p-3}/(p-3)-B_{2p-4}/(4p-8)$. Now combining \eqref{Sigma2-1}--\eqref{h2-1}, Lemmas \ref{harmonic} and \ref{h21} we obtain that
\begin{equation}\label{Sigma2}
\begin{aligned}
\Sigma_2\eq&\f{H_{(p-1)/2}^{(2)}}{p}+\f{21H_{p-1}}{2p^2}\pmod{p^2}.
\end{aligned}
\end{equation}
Substituting \eqref{Sigma1} and \eqref{Sigma2} into \eqref{key} and in light of \eqref{key'} and Lemma \ref{harmonic} we have
\begin{align*}
\sum_{m=0}^{p-1}(2m+1)A_m\eq& p-2p^3H_{p-1}^{(2)}-\f{3}{2}p^3H_{(p-1)/2}^{(2)}+\f{1}{2}p^4H_{(p-1)/2}^{(3)}-21p^2H_{p-1}\\
\eq&p-\f{7}{2}p^2H_{p-1}\pmod{p^6}.
\end{align*}
The proof of Theorem \ref{th1} is complete now.\qed

\section{Proof of Theorem \ref{th2}}
\setcounter{lemma}{0}
\setcounter{theorem}{0}
\setcounter{equation}{0}
\setcounter{conjecture}{0}
\setcounter{proposition}{0}

In order to show Theorem \ref{th2}, we need the following results.
\begin{lemma}\label{1kh22}
Let $p\geq7$ be a prime. Then we have
$$
\sum_{k=1}^{p-1}\f{H(2,2;k)}{k}\eq-\f{1}{2}B_{p-5}\pmod{p}.
$$
\end{lemma}
\begin{proof}
By Remark 2.1 we have
\begin{align*}
\sum_{k=1}^{p-1}\f{H(2,2;k)}{k}=\sum_{k=1}^{p-1}\f{1}{k}\sum_{1\leq i<j\leq k}\f{1}{i^2j^2}\eq-\sum_{1\leq i<j\leq p-1}\f{H_{j-1}}{i^2j^2}=-\sum_{j=1}^{p-1}\f{H_{j-1}H_{j-1}^{(2)}}{j^2}\pmod{p}.
\end{align*}
On one hand,
\begin{equation}\label{j-1j-1}
\sum_{j=1}^{p-1}\f{H_{j-1}H_{j-1}^{(2)}}{j^2}=\sum_{j=1}^{p-1}\f{H_jH_j^{(2)}-H_j^{(2)}/j-H_j/j^2+1/j^3}{j^2}.
\end{equation}
On the other hand, we have
$$
\sum_{j=1}^{p-1}\f{H_{j-1}H_{j-1}^{(2)}}{j^2}=\sum_{j=1}^{p-1}\f{H_{p-j-1}H_{p-j-1}^{(2)}}{(p-j)^2}\eq-\sum_{j=1}^{p-1}\f{H_{j}H_{j}^{(2)}}{j^2}\pmod{p}
$$
in view of that $H_{p-1-k}\eq H_k\pmod{p}$ and $H_{p-1-k}^{(2)}\eq-H_k^{(2)}\pmod{p}$. Combining this with \eqref{j-1j-1} we arrive at
\begin{align*}
\sum_{j=1}^{p-1}\f{H_{j-1}H_{j-1}^{(2)}}{j^2}\eq\f{1}{2}\l(-H_{p-1}^{(5)}-H(2,3;p-1)-H(1,4;p-1)\r)\pmod{p}.
\end{align*}
By \cite[Theorem 3.1]{TZ} we have $H(2,3;p-1)\eq-2B_{p-5}\pmod{p}$ and $H(1,4;p-1)\eq B_{p-5}\pmod{p}$ provided that $p\geq7$. These together with Lemma \ref{harmonic} imply that
$$
\sum_{j=1}^{p-1}\f{H_{j-1}H_{j-1}^{(2)}}{j^2}\eq\f{1}{2}B_{p-5}\pmod{p}.
$$
This proves the desired Lemma \ref{1kh22}.
\end{proof}

\begin{lemma}\label{h-k2k}
For any prime $p\geq7$ we have
\begin{equation}\label{lemmah-k2k}
\sum_{k=1}^{p-1}\f{H_k^{(2)}}{k}\eq\f{3H_{p-1}}{p^2}-\f{1}{2}p^2B_{p-5}\pmod{p^3}.
\end{equation}
\end{lemma}
\begin{proof}
Letting $n=(p-1)/2$ in \cite[Lemma 3.1]{LW} we obtain that
$$
\sum_{k=1}^{p-1}\f{(-1)^k}{k}\binom{p-1}{k}\binom{p+k}{k}=-2H_{p-1}.
$$
Note that
\begin{align*}
\binom{p-1}{k}\binom{p+k}{k}=&\prod_{j=1}^k\f{p^2-j^2}{j^2}=(-1)^k\prod_{j=1}^k\l(1-\f{p^2}{j^2}\r)\\
\eq&(-1)^k\l(1-p^2H_k^{(2)}+p^4H(2,2;k)\r)\pmod{p^5}.
\end{align*}
It follows that
$$
H_{p-1}-p^2\sum_{k=1}^{p-1}\f{H_k^{(2)}}{k}+p^4\sum_{k=1}^{p-1}\f{H(2,2;k)}{k}\eq-2H_{p-1}\pmod{p^5}.
$$
With the help of Lemma \ref{1kh22}, we obtain \eqref{lemmah-k2k}.
\end{proof}

\begin{lemma}\label{theorem2lemma1} For any prime $p\geq7$ we have
\begin{gather}
\label{h22k}\sum_{k=0}^{p-1}H(2,2;k)\eq-\f{p}{2}H_{p-1}^{(4)}-\f{3H_{p-1}}{p^2}+H_{p-1}^{(3)}+\f{1}{2}p^2B_{p-5}\pmod{p^3},\\
\label{h2442}\sum_{k=0}^{p-1}(H(2,4;k)+H(4,2;k))\eq3B_{p-5}\pmod{p},\\
\label{h222}\sum_{k=0}^{p-1}H(2,2,2;k)\eq-\f{3}{2}B_{p-3}\pmod{p}.
\end{gather}
\end{lemma}
\begin{proof}
By Lemma \ref{harmonic} we arrive at
\begin{align*}
\sum_{k=0}^{p-1}H(2,2;k)=&\sum_{k=1}^{p-1}\sum_{1\leq i<j\leq k}\f{1}{i^2j^2}=\sum_{1\leq i<j\leq p-1}\f{p-j}{i^2j^2}\\
=&\f{p}{2}\l(\l(H_{p-1}^{(2)}\r)^2-H_{p-1}^{(4)}\r)-\sum_{1\leq i<j\leq p-1}\f{1}{i^2j}\\
\eq&-\f{p}{2}H_{p-1}^{(4)}+\sum_{k=1}^{p-1}\f{H_k}{k^2}\pmod{p^3}.
\end{align*}
Furthermore,
\begin{align*}
\sum_{k=1}^{p-1}\f{H_k}{k^2}=&\sum_{k=1}^{p-1}\f{1}{k^2}\sum_{j=1}^{k}\f{1}{j}=\sum_{j=1}^{p-1}\f{1}{j}\sum_{k=j}^{p-1}\f{1}{k^2}\\
\eq&-\sum_{j=1}^{p-1}\f{H_j^{(2)}}{j}+H_{p-1}^{(3)}\pmod{p^3}.
\end{align*}
From the above and with the help of Lemma \ref{h-k2k}, we obtain \eqref{h22k}.

Now we turn to prove \eqref{h2442}. It is easy to see that
\begin{align*}
\sum_{k=0}^{p-1}(H(2,4;k)+H(4,2;k))=&\sum_{1\leq i<j\leq p-1}\f{p-j}{i^2j^4}+\sum_{1\leq i<j\leq p-1}\f{p-j}{i^4j^2}\\
\eq&-H(2,3;p-1)-H(4,1;p-1)\pmod{p}.
\end{align*}
By \cite[Theorem 3.1]{TZ} we have $H(2,3;p-1)\eq-2B_{p-5}\pmod{p}$ and $H(4,1;p-1)\eq -B_{p-5}\pmod{p}$ for $p\geq7$. Then \eqref{h2442} follows at once.

Finally, we consider \eqref{h222}. Clearly,
\begin{align*}
\sum_{k=0}^{p-1}H(2,2,2;k)=\sum_{1\leq i_1<i_2<i_3\leq p-1}\f{p-i_3}{i_1^2i_2^2i_3^2}\eq -H(2,2,1;p-1)\pmod{p}.
\end{align*}
By \cite[Theorem 3.5]{Zhao}, we have
$$
H(2,2,1;p-1)\eq\f{3}{2}B_{p-3}\pmod{p}.
$$
The proof of Lemma \ref{theorem2lemma1} is now complete.
\end{proof}

\begin{lemma}\label{identity2} Let $k\in\N$. Then for $n\in\Z^+$ we have
$$
\sum_{m=0}^{n-1}(2m+1)^3\binom{m+k}{2k}^2=\f{(n-k)^2(2n^2-k-1)}{k+1}\binom{n+k}{2k}^2.
$$
\end{lemma}
\begin{proof}
It can be verified directly by induction on $n$.
\end{proof}

\noindent {\it Proof of Theorem \ref{th2}}. The case $p=5$ can be verified directly. Below we assume that $p\geq7$. By Lemma \ref{identity2} we have
\begin{align*}
\sum_{m=0}^{p-1}(2m+1)^3A_m=&\sum_{m=0}^{p-1}(2m+1)^3\sum_{k=0}^m\binom{m+k}{2k}^2\binom{2k}{k}^2=\sum_{k=0}^{p-1}\binom{2k}{k}^2\sum_{m=0}^{p-1}(2m+1)^3\binom{m+k}{2k}^2\\
=&\sum_{k=0}^{p-1}\binom{2k}{k}^2\f{(p-k)^2(2p^2-k-1)}{k+1}\binom{p+k}{2k}^2\\
=&p^2\sum_{k=0}^{p-1}\f{2p^2-k-1}{k+1}\binom{p-1}{k}^2\binom{p+k}{k}^2.
\end{align*}
Noting that
\begin{align*}
&\binom{p-1}{k}^2\binom{p+k}{k}^2=\prod_{j=1}^k\l(1-\f{p^2}{j^2}\r)^2\eq\prod_{j=1}^k\l(1-\f{2p^2}{j^2}+\f{p^4}{j^4}\r)\\
\eq&1-2p^2H_k^{(2)}+p^4H_k^{(4)}+4p^4H(2,2;k)-2p^6\l(H(2,4;k)+H(4,2;k)\r)\\
&-8p^6H(2,2,2;k)\pmod{p^7},
\end{align*}
we arrive at
\begin{align*}
\sum_{m=0}^{p-1}(2m+1)^3A_m\eq&2p^4\sum_{k=1}^{p-1}\f{1}{k+1}\l(1-2p^2H_k^{(2)}+p^4H_{k}^{(4)}+4p^4H(2,2;k)\r)\\
&-p^2\sum_{k=0}^{p-1}\bigg(1-2p^2H_k^{(2)}+p^4H_k^{(4)}+4p^4H(2,2;k)\\
&-2p^6\l(H(2,4;k)+H(4,2;k)\r)-8p^6H(2,2,2;k)\bigg)\pmod{p^9}.
\end{align*}
It is clear that
\begin{equation*}
\sum_{k=0}^{p-1}\f{1}{k+1}=H_{p-1}+\f{1}{p}.
\end{equation*}
With the help of Lemma \ref{h-k2k} we obtain that
\begin{equation*}
\begin{aligned}
\sum_{k=0}^{p-1}\f{H_k^{(2)}}{k+1}=&\sum_{k=1}^{p}\f{H_{k-1}^{(2)}}{k}=\sum_{k=1}^{p-1}\f{H_k^{(2)}}{k}-H_{p-1}^{(3)}+\f{H_{p-1}^{(2)}}{p}\\
\eq&\f{3}{p^2}H_{p-1}+\f{H_{p-1}^{(2)}}{p}-H_{p-1}^{(3)}-\f{1}{2}p^2B_{p-5}\pmod{p^3}.
\end{aligned}
\end{equation*}
Clearly,
$$
\sum_{k=0}^{p-1}\f{H_k^{(4)}}{k+1}=H(4,1;p-1)+\f{H_{p-1}^{(4)}}{p}\eq-B_{p-5}+\f{H_{p-1}^{(4)}}{p}\pmod{p}.
$$
Furthermore,
\begin{align*}
\sum_{k=0}^{p-1}\f{H(2,2;k)}{k+1}=&\f{1}{2}\sum_{k=1}^{p}\f{1}{k}\l(\l(H_{k-1}^{(2)}\r)^2-H_{k-1}^{(4)}\r)=\f{1}{2}\sum_{k=1}^{p}\f{1}{k}\l(\l(H_{k}^{(2)}\r)^2-H_{k}^{(4)}-\f{2H_k^{(2)}}{k^2}+\f{2}{k^4}\r)\\
=&\sum_{k=1}^{p-1}\f{H(2,2;k)}{k}-H(2,3;p-1)+\f{1}{2p}\l(\l(H_{p-1}^{(2)}\r)^2-H_{p-1}^{(4)}\r).
\end{align*}
Then by Lemmas \ref{harmonic} and \ref{1kh22} we arrive at
$$
\sum_{k=0}^{p-1}\f{H(2,2;k)}{k+1}\eq\f{3}{2}B_{p-5}-\f{1}{2p}H_{p-1}^{(4)}\pmod{p}.
$$
For $r=2,4$ we have
\begin{equation*}
\sum_{k=0}^{p-1}H_k^{(r)}=\sum_{k=1}^{p-1}\sum_{l=1}^k\f{1}{l^r}=\sum_{l=1}^{p-1}\f{p-l}{l^r}=pH_{p-1}^{(r)}-H_{p-1}^{(r-1)}.
\end{equation*}
Combining the above and in view of Lemma \ref{theorem2lemma1} we obtain
\begin{align*}
\sum_{m=0}^{p-1}(2m+1)^3A_m\eq&2p^4H_{p-1}+2p^3-12p^4H_{p-1}+4p^6H_{p-1}^{(3)}-4p^5H_{p-1}^{(2)}+2p^8B_{p-5}-2p^8B_{p-5}\\
&+2p^7H_{p-1}^{(4)}+12p^8B_{p-5}-4p^7H_{p-1}^{(4)}-p^3+2p^5H_{p-1}^{(2)}-2p^4H_{p-1}-p^7H_{p-1}^{(4)}\\
&+p^6H_{p-1}^{(3)}+2p^7H_{p-1}^{(4)}+12p^4H_{p-1}-2p^8B_{p-5}-4p^6H_{p-1}^{(3)}-6p^8B_{p-5}\\
=&p^3-2p^4H_{p-1}-2p^5H_{p-1}^{(2)}+p^6H_{p-1}^{(3)}-p^7H_{p-1}^{(4)}+4p^8B_{p-5}\pmod{p^9}.
\end{align*}
Then Theorem \ref{th2} follows from Lemmas \ref{cms} and \ref{harmonic}.\qed

\end{document}